\documentclass[12pt, leqno]{amsart}
\usepackage{amsmath}
\usepackage{amssymb}
\usepackage{amsthm}
\usepackage{enumerate}
\usepackage[mathscr]{eucal}
\usepackage{xcolor}
\theoremstyle{plain}
\usepackage{tikz}
\usepackage{soul}
\usepackage[normalem]{ulem}
\usepackage[normalem]{ulem}
\newtheorem{theorem}{Theorem}[section]
\newtheorem{lemma}[theorem]{Lemma}
\newtheorem{prop}[theorem]{Proposition}
\newtheorem{cor}[theorem]{Corollary}

\theoremstyle{definition}
\newtheorem{definition}[theorem]{Definition}
\newtheorem{remark}[theorem]{Remark}

\newtheorem{example}[theorem]{Example}

\theoremstyle{remark}

\newcommand{\Ext}{{\rm Ext}\,}

\newcommand{\sm}{{\rm Sm}\,}

\numberwithin{equation}{section}
\allowdisplaybreaks

\begin{document}
	
\title[Simultaneous Zero Inclusion property]{On the simultaneous zero inclusion property and characterization of Inner product space}
	\author[S. Ghosh]{Souvik Ghosh}
    
	\address[Ghosh]{Department of Mathematical Sciences, Indian Institute of Science Education and		Research (IISER) Mohali, Knowledge City, Sector 81, SAS Nagar, Punjab-140306, India}	  \email{sghosh0019@gmail.com, souvikghosh@iisermohali.ac.in}
	
	
	
	\begin{abstract}
		We show that if a Banach space $X$ is smooth and almost locally uniformly rotund and if satisfies the following $$0\in W(T) \implies 0\in W(T^{-1})$$ for all invertible bounded linear map on $X$ then Birkhoff-James orthogonality in $X$ is symmetric.
 This result, therefore,  characterizes the inner product spaces out of all smooth and almost locally uniformly rotund spaces having three or more dimension via simultaneous zero inclusion property, introduced by Bra\v{c}i\v{c} et al. [J. Math. Anal. Appl., \textbf{449} (2017) 1413–1423]. Further, this gives answers to some open questions asked  by E. Poon [J. Math. Anal. Appl., \textbf{534 }(2024) 128052].
	\end{abstract}
	
	\maketitle
	
	\section{Introduction.} 
The numerical range is one of the fundamental notions in operator theory, providing a useful connection between the algebraic and geometric properties of bounded linear operators. In the setting of Hilbert spaces, the classical numerical range enjoys several remarkable features, among which the celebrated Toeplitz-Hausdorff theorem asserts that it is always a convex subset of the complex plane. This convexity has profound consequences and has motivated extensive research in operator theory and related areas. In contrast, when one considers bounded linear operators on a general real or complex Banach space, the situation becomes considerably more intricate. \\
Let $X$ be a complex Banach space with unless otherwise stated and let $X^*$ denote its dual space. $B_X$ and $S_X$ denote as the unit ball and unit sphere of $X$, respectively. For a convex bounded subset $A$ of $X,$ we denote $\Ext A$ as the collection of all extreme points of $A.$ Given $x\in S_X,$ $J(x)=\{x^*\in S_{X^*}: x^*(x)=1\}$ is the collection of all the supporting functionals of $x.$ When $J(x)$ is singleton $x$ is \emph{smooth} or the norm is G\^ateaux differentiable at $x.$ We denote $\sm(X)$ as the collection of all smooth points in $X$. $X$ is smooth if every $x\in S_X$ is smooth, i.e., $\sm(X)=X\setminus\{0\}$. An element $x\in S_X$ is said to be \emph{exposed} if there exists $x^*\in J(x)$ such that for all $y\in S_X,$ $x^*(y)=1\implies y=x.$ $x\in S_X$ is called rotund if for any $y\in S_X,$ $\frac{\|x+y\|}{2}=1 \implies x=y.$ In a
strictly convex space every element of the unit sphere is rotund. An element $x\in S_X$ is said to be \emph{almost locally uniformly rotund (ALUR)} if given any $\{x_m^*\}\subset S_{X^*}, \{x_n\}\subset S_X,$ 
$$\lim_{m\to \infty}\lim_{n\to \infty} x_m^*\frac{(x_n+x)}{2}=1$$ then it implies that $\lim\limits_{n\to\infty} x_n=x.$ The space $X$ is said to be ALUR if each element of $S_X$ is ALUR.   $L(X)$ denotes the space of all bounded linear operators on $X.$ Given $T\in L(X)$, $M_T$ is the set of all norm attainment set of $T,$ i.e., $M_T:=\{x\in S_X: \|Tx\|=\|T\|\}.$ The collection of all invertible operators in $L(X)$ is denoted as $GL(X).$ Define $$
\Pi=\{(x,x^*)\in X\times X^*: x\in S_X, x^*\in J(x)\}.$$
Then, for $T\in L(X)$, the (spatial) numerical range of $T$ is defined by $$ W(T)=\{\,x^*(Tx):(x,x^*)\in\Pi\,\}.$$ The notion of the spatial numerical range naturally extends the classical concept, but many of its appealing geometric properties, including convexity, need not persist unless the norm is induced by an inner product.  When $X$ is a Hilbert space, this definition coincides with the classical numerical range (or field of values). These attracts many mathematicians to  understand the geometric and analytical properties of numerical ranges in Banach spaces and their relationship with the geometry of the underlying space.

The concept of simultaneous zero inclusion property was first introduced by Bra\v ci\v c and Diogo in \cite{Bracic} to understand the spatial numerical range.  
    \begin{definition}
        Let $T\in L(X)$ be invertible. Then $X$ has \textit{simultaneous zero inclusion property (in short, $S0I$)} at $T$ if $0\in W(T) \implies 0\in W(T^{-1})$. We say $X$ has \textit{$S0I$} if for each invertible $T\in L(X)$, $0\in W(T)\implies 0\in W(T^{-1}).$
    \end{definition}
    Bra\v ci\v c et al. studied the $S0I$ property in finite-dimensional complex Banach spaces. They showed that neither $\ell_1^n$ nor $\ell_\infty^n$ has the $S0I$ property. Later, Poon \cite{Poon} further investigated this property. For an absolute symmetric norm on $\mathbb C^2$, he proved that the $S0I$ property holds if and only if the norm is induced by an inner product. He also established a connection between the $S0I$ property and Birkhoff-James orthogonality.
\\
The concept of Birkhoff--James orthogonality provides a natural extension of the classical notion of orthogonality from inner product spaces to arbitrary normed spaces. For $x,y\in X$, we say that $x$ is \emph{Birkhoff-James orthogonal} to $y$, denoted by $x\perp_{B} y$, if
$$
\|x+\lambda y\|\geq \|x\| \quad \text{for every } \lambda\in\mathbb{F},
$$
where $\mathbb{F}$ denotes the underlying scalar field. This notion was first introduced by Birkhoff \cite{Birkhoff35} and later studied extensively by James \cite{James45, James47a, James47b}, and it has become one of the central concepts in the geometry of Banach spaces. In this regard, we note the following chracterization given by James which will be an essential tool for the whole article.
\begin{theorem}\label{BJ_James}
    Let $X$ be a normed space and let $x, y\in X.$ Then $x\perp_B y$ if and only if there exists $x^*\in J(x)$ such that $x^*(y)=0.$
\end{theorem}
Unlike orthogonality in Hilbert spaces, Birkhoff-James orthogonality is, in general, not symmetric. In particular, $x\perp_{B} y$ does not necessarily imply $y\perp_{B} x$. A fundamental characterization due independently to James and Day shows that the symmetry of Birkhoff-James orthogonality completely characterizes inner product spaces among Banach spaces.

\begin{theorem}\cite{Day47, James47a}\label{symmetry}
Let $X$ be a Banach space with $\dim(X)\geq 3$. Then Birkhoff-James orthogonality is symmetric, i.e., for all $x, y\in X$, $x\perp_B \implies y\perp_B x$ if and only if the norm on $X$ is induced by an inner product. 
\end{theorem}
A local version of symmetry of Birkhoof-James orthogonality was introduced and studied by Sain in \cite{Sain17}. For $x\in X$ is said to be l\textit{eft-symmetric (right symmetric)}  if for any $y\in X$ $x\perp_B y\implies y\perp_B x$ ($y\perp_B x\implies x\perp_B y$). Given a nonempty subset $\mathcal{A}$ of $X$, we say $x$ is \textit{local-$\mathcal{A}$ left symmetric (right symmetric)} if for all $y\in \mathcal{A}$, $x\perp_B y\implies y\perp_B x$ ($y\perp_B x\implies x\perp_B y$). Readers can go through the book \cite[Chapter 7]{MPS} for a detailed overview on local symmetry of Birkhoff-James orthogonality. \\ 
 Poon \cite[Prop. 1.6]{Poon} observed the following result which connects $S0I$ property with Birkhoff-James orthogonality.
 \begin{theorem}\label{connection; Poon}
     A normed space $X$ has the $S0I$ property if and only if, for every invertible $A \in L(X),$  and $x\neq 0$
and $x\perp_B Ax$ implies there exists  $z \neq 0$ such that $Az\perp_B  z.$	
 \end{theorem}   
It immediately shows that if Birkhoff-James orthogonality is symmetric then $X$ has $S0I$. \emph{But what about the converse?}\\
In this article, we mainly investigate this particular question. We show that, for smooth and almost locally uniformly rotund spaces, the $S0I$ property imposes strong restrictions on the underlying norm. Our main result is the following:

\begin{theorem}\label{main}
    Let $X$ be a smooth and ALUR normed space. Then  $X$ has the simultaneous zero inclusion property ($S0I$) if and only if Birkhoff-James orthogonality in $X$ is symmetric, i.e., for all $x, y\in X$, $x\perp_B y\implies y\perp_B x.$
\end{theorem}
 It can be readily seen that Theorem \ref{symmetry} and Theorem \ref{main} together shows: \\
\emph{If $X$ is a smooth and ALUR Banach space with $\dim X\geq 3$ then $X$ has $S0I$ if and only if the norm of $X$ is induced by an inner product.}\\
As  a consequence of our technique, we investigate the $S0I$ in finite-dimensional settings and obtained some stronger results. In the last section we discuss that in a real finite-dimensional Banach space, some local symmetry in $L(X)$ guarantees the $S0I$ property for the invertible operators.

    \section{Some preliminary observations.}
   
 We begin this section with some basic observations follows from the connection between $S0I$ and Birkhoff-James orthogonality. For that matter, we define the following two sets.  Let $T\in L(X)$ and we consider the following two sets:
        \[\mathcal{P}_T=\{x\in S_X: x\perp_B Tx\};\\
        \mathcal{Q}_T=\{y\in S_X: Ty\perp_B y\}.\]

	One can easily see that if $T$ is not invertible then $\mathcal{P}_T$ and $\mathcal{Q}_T$ are always nonempty. In fact, $\ker T\subset \mathcal{P}_T\cap \mathcal{Q}_T.$ In case of invertible $T$, the sets can be empty. For example, let us take two-dimensional Euclidean space $\mathbb{R}^2$ and let $T=\begin{pmatrix}
	    \frac{1}{\sqrt{2}} &  \frac{1}{\sqrt{2}}\\ - \frac{1}{\sqrt{2}}  &  \frac{1}{\sqrt{2}}
	\end{pmatrix}.$  In this case, $\langle Tx, x\rangle=\frac{\|Tx\|}{\sqrt{2}}\neq 0.$  In this case  $\mathcal{P}_T=\mathcal{Q}_T=\emptyset.$ In the next proposition we write how the set $\mathcal{P}_T, \mathcal{Q}_T$ can be used to characterize the $S0I$ property, following from Theorem \ref{connection; Poon}. 
  
	\begin{prop}\label{prop1}
	    Let $X$ be a Banach space and let $T\in L(X)$ be invertible. Then $X$ has $S0I$ at $T$ if and only if either of the following holds true:
        \begin{itemize}
            \item [(i] $\mathcal{P}_T=\emptyset$ as well as $\mathcal{Q}_T=\emptyset.$
            \item[(ii)] Both $\mathcal{P}_T$ and $\mathcal{Q}_T$ are nonempty.  
        \end{itemize}
	\end{prop}
	
From the above proposition we would like to note the following useful observation.
\begin{cor}
    Let $X$ be a Banach space. Then $X$ does not have $S0I$ if and only if there exists an invertible $T\in L(X)$ such that $\mathcal{P}_T\neq \emptyset$ but $\mathcal{Q}_T=\emptyset.$ 
\end{cor}

  Let us denote $\mathcal{L}_X$ and $\mathcal{R}_X$ as the collection of all left and right symmetric elements of $X,$ respectively. In particular, \[ \mathcal{L}_X=\{x\in X: x\perp_B y\implies y\perp_B x, \forall y\}; \quad \mathcal{R}_X=\{x\in X: y\perp_B x\implies x\perp_B y, \, \forall y\}.\]
  Considering these elements, we observe the following. We do not include the proof as it follows easily. 
    \begin{prop}\label{prop; $S0I$}
        Let $T\in L(X)$ be invertible. Suppose that  either of the following holds true:
        \begin{itemize}
            \item [(i)] $\mathcal{P}_T\cap \mathcal{L}_X \neq \emptyset$;
            \item [(ii)] $\mathcal{Q}_T\cap \mathcal{R}_X\neq \emptyset$;
            \item[(iii)] $T(\mathcal{P}_T)\cap \mathcal{R}_{T(X)}\neq \emptyset$;
            \item[(iv)] $T(\mathcal{Q}_T)\cap \mathcal{L}_{T(X)}\neq \emptyset$.
        \end{itemize}
        Then $X$ has $S0I$ at $T.$ Consequently, if it holds for every invertible $T\in L(X)$ then $X$ has $S0I$.
    \end{prop}

\begin{proof}
We only show (i) and (iii) as (ii) and (iv) follows similarly.\\
    (i) Suppose that $x\in \mathcal{P}_T\cap \mathcal{L}_X.$ Then $x\perp_B Tx$. Also, $x$ being left symmetric we have $Tx\perp_B x.$ This implies $\mathcal{P}_T\neq \emptyset$ as well as $\mathcal{Q}_T\neq\emptyset.$ Thus from Proposition \ref{prop1} we get $X$ has $S0I$ at $T.$  \\
    (iii) Let $y\in T(\mathcal{P}_T)\cap \mathcal{R}_{T(X)}.$ Then $y=Tx$ for some $x\in \mathcal{P}_T.$ Since $Tx$ is right symmetric,  $Tx\perp_B x$ and therefore, $\mathcal{Q}_T\neq \emptyset.$ This proves again from Proposition \ref{prop1} that $X$ has $S0I$. 
\end{proof}

The above proposition can help us to find some invertible operators with $S0I$ property.
\begin{example}
    Let $X=\ell_p^3$ with $1<p<\infty$ and let $T\in L(\ell_p^3)$ be the following with respect to standard ordered basis:
    \[T=\begin{pmatrix}
        0 & 2 & 1\\
        1 & 1 & 0\\
        1 & 0 & 2
    \end{pmatrix}\]
    Clearly, $T$ is invertible. Since $Te_1 \in span\{e_2, e_3\},$ it follows that $e_1\perp_B Te_1.$ Moreover, from \cite{CSS} we have $e_1$ is a left symmetric point in $\ell_p^3$ Thus we have $e_1\in \mathcal{P}_T\cap \mathcal{L}_X.$ This implies from Proposition \ref{prop; $S0I$} that $X$ has $S0I$ at $T.$
\end{example}

\section{Main results.}
To prove the main theorem of this article, we will construct an invertible bounded linear operator on $X $ in such a way that it does not have $S0I$ property.  We  need the following geometric property of Banach space which is immediate from the definition of exposed points.
\begin{lemma}\label{lem; smooth and exposed}
   Let $x \in X$ be such that $x$ is smooth as well as exposed. If for any $v\in S_X,$  $J(x)\cap J(v)\neq \emptyset$ then $x=v.$ 
\end{lemma}
Let us now prove the main theorem.\\
\noindent \textbf{Proof of Theorem \ref{main}.}
To show  the sufficient part suppose first that $T\in GL(X)$ is such that $0\in W(T).$ Then we have $x\perp_B Tx$ for some $x\in S_X.$ Since Birkhoff-James orthogonality is symmetric, it follows that $Tx\perp_B x.$ This follows from Proposition \ref{prop1} that $0\in W(T^{-1})$. This shows that $X$ has $S0I$. \vspace{.2cm}

\noindent Next we prove the necessary part. Suppose that $X$ has $S0I$, and on the contrary, there exist $x,y\in S_X$ such that $x\perp_B y$ but $y\not\perp_B x$. Since $X$ is smooth, let $J(y)=\{y^*\}$. Following Theorem \ref{BJ_James},  we have
$\xi:=y^*(x)\neq0$. Without loss of generality we may assume that $Re \, \xi>0.$ Consider the following two functions $h, g: S_X\to \mathbb C$ defined as 
$$h(v)=y^*(v)v^*(x) \,\, \mbox{and} \,\, g(v)=1-y^*(v)v^*(y).$$  Notice that $Re\,g(v)\geq0$ for every $v\in S_X$, since $|y^*(v)|\leq1$ and $|v^*(y)|\leq1$. Moreover, if $g(v)=0$, then we get $y^*(v)v^*(y)=1$. This implies  $y^*(v)=\alpha, v^*(y)=\overline{\alpha}$, for some $|\alpha|=1.$ It follows that $J(\overline{\alpha}v)=\{y^*\}=J(y)$ i.e., $J(\overline{\alpha}v)\cap J(y)\neq \emptyset$. Since $y$ is smooth and exposed, it follows from Lemma \ref{lem; smooth and exposed} that $v=\alpha y$, and consequently $h(v)=\alpha \overline{\alpha} y^*(x)=\xi\neq0$. Thus we have $g(v)=0\Longrightarrow h(v)=\xi\neq0.$  Note that $h, g$ both are continuous map. Indeed, let $\{v_n\}\subset S_X$ be such that $v_n\to v$ with $J(v_n)=\{v_n^*\}\subset S_{X^*}$. Since $B_{X^*}$ is weak$^*$-compact, there exists, $z^*\in B_{X^*}$ such that $v_n^*\overset{w^*}{\rightarrow} z^*.$ It is straightforward to see that $z^*(v)=1$ and therefore, $z^*=v^*.$
\begin{eqnarray*}
|h(v_n)-h(v)|&=& |y^*(v_n)v_n^*(x)-y^*(v)v^*(x)|\\
&=& |y^*(v_n)v_n^*(x)-y^*(v)v_n^*(x)+ y^*(v)v_n^*(x)- y^*(v)v^*(x)|\\
&\leq& |y^*(v_n)v_n^*(x)-y^*(v)v_n^*(x)|+ |y^*(v)v_n^*(x)- y^*(v)v^*(x)|\\
&\leq& \|y^*\|\|v_n-v\||v_n^*(x)| + |y^*(v)||v_n^*(x)-v^*(x)|.
 \end{eqnarray*}
This inequality shows that $h$ must be continuous. We can show the continuity of $g$ with similar argument. Let $K=\{v\in S_X:g(v)=0\}=\{\alpha y: \alpha\in \mathbb C, |\alpha|=1\}$. By the preceding observation, $h(v)=\xi$ for all $v\in K$. Since $K$ is a compact subset of $S_X$ and $h$ is continuous, we can consider the open set $U=Re\, h^{-1}(\frac{Re\, \xi}{2}, \infty)\cap S_X$ with respect to the subspace topology such that for any $v\in U,$ $Re \, h(v)>\frac{Re \, \xi}{2}>0.$
Now we claim that on the set $S_X\setminus U$, we have $Re\, g(v)\geq c>0$. If possible let $\inf\{Re\, g(v): v\in S_X\setminus U\}=0.$ This implies there exists $\{v_n\}\subset S_X\setminus U$ such that $Re\,g(v_n)\to 0.$ So we have $y^*(v_n)v_n^*(y)\to 1.$ It forces $Re\,y^*(\beta v_n)\to 1$ as well as $Re\,v_n^*(\overline{\beta}y)\to 1,$ for some $\beta\in \mathbb C$ and $|\beta|=1.$ From the last limit we can write $\lim\limits_{n\to \infty} Re\, v_n^*\bigg(\frac{ v_n+ \overline{\beta} y}{2}\bigg)=1$  Since $X$ is ALUR, we get $ v_n\to \overline{\beta}y$ and therefore, there exists $N\in \mathbb N$ such that for all $n\geq N,$ $v_n\in U.$ This is not possible from our assumption. This proves our claim. Hence there exists $c>0$ such that $Re\, g(v)\geq c$ on $S_X\setminus U$. Also put 
$M=\sup_{v\in S_X}|h(v)|\leq 1$ and let $|\lambda|>M/c$. Then, for $v\in S_X\setminus U$,
$|\lambda g(v)|>M\geq|h(v)|$,
and consequently $Re\, h(v)+\lambda Re\, g(v)\neq0$. For $v\in U$, both $Re\, h(v)>0$ and $Re\, \lambda g(v)>0$, so again $Re\, (h(v)+\lambda g(v))>0$. Now we define the operator on $X$ as: 
\[S_\lambda=(x-\lambda y)\otimes y^*+\lambda I.\]
From the above arguments  we have obtained  that $v^*(S_\lambda v)=h(v)+\lambda g(v)\neq0$ for every $v\in S_X$. Thus we have $S_\lambda$ is one-one. Now suppose that $w\in X$ be any nonzero element. We choose $v= \frac{1}{\lambda}w-\frac{y^*(w)}{\lambda a}(x-\lambda y).$ By a straightforward calculation one can show that $S_\lambda(v)=w.$ This proves that $S_\lambda$ is onto and consequently, it is bijective. We assume that $S_\lambda^{-1}=T.$ Note that $Tx=S_\lambda^{-1}(x)=y$ and since $x\perp_B y$, we get $x\perp_B Tx.$ On the other hand, for each $z\in S_X$ such that $z^*(S_\lambda z)\neq 0$, i.e.,  $z\not\perp_B S_{\lambda}z,$ i.e., $z\not\perp_B T^{-1}z$ for any $z\in S_X.$ This shows that $0\notin W(T^{-1})$ and contradicts the fact that $X$ has $S0I$ at $T.$ Therefore, Birkhoff-James orthogonality must be symmetric. This completes the proof of the theorem. \hfill $\square$\\ 

In the following result we show that for the case of finite-dimensional Banach space only smoothness (or, strict convexity) is sufficient to prove the above theorem. Before that we note the following lemma.

\begin{lemma}\label{Dual; Radon}
    Let $X$ be a Radon plane then $X^*$ is also a Radon plane.
\end{lemma}
\begin{proof}
 Suppose that $x^*\perp_B y^*$ for some $x^*,y^*\in S_{X^*}$. By Theorem \ref{BJ_James}, there exists $x\in S_X$ such that $|x^*(x)|=1,$ and   $y^*(x)=0.$ Choose $y\in S_X$ with $x^*(y)=0$. Since $x^*\in J(x)\cup J(-x)$, Theorem \ref{BJ_James} gives $x\perp_B y$. As $X$ is a Radon plane, $y\perp_B x$. Hence, again by Theorem \ref{BJ_James}, there exists $z^*\in J(y)$ such that $z^*(x)=0$. Since $y^*(x)=0$ and $y^*(y)=1$, we have $z^*=c  y^*$. Therefore, $y^*\perp_B x^*,$ and thus $X^*$ is a Radon plane.

\end{proof}
\begin{theorem}\label{finite; smooth}
     Let $X$ be a finite-dimensional smooth or strictly convex Banach space. Then $X$ has $S0I$ if and only if Birkhoff-James orthogonality is symmetric in $X.$  
\end{theorem}

\begin{proof}
Since the sufficient part follows from Theorem \ref{connection; Poon} directly, we show the necessary part.
As $X$ is smooth, we can use the arguments given in the proof of Theorem \ref{main} and obtain that there exists an open set $U$ of $S_X$ containing $K$ such that for all $v\in U$ $Re\, h(v)>0$. Also, $Re\, g(v)\geq 0$ for all $v\in U.$ Here only we need to show that  there exists $c>0$ such that $Re \, g(v)\geq c>0$ for all $v\in S_X\setminus U.$  As $S_X\setminus U$ is compact,  $Re\, g(v)>0$ for all $v\in S_X\setminus U$ and this implies there exists $c>0$ such that $c=\min_{v\in S_X\setminus U} Re\, g(v).$ Then we can construct the same $S_\lambda$ to prove the desired result.\\
On the other hand, if $X$ is  strictly convex then $X^*$ is smooth. As $X$ has $S0I$,  we note from \cite[Th. 2.2]{Bracic} that $X^*$ also has $S0I$. Following the above argument we get that Birkhoff-James orthogonality in $X^*$ is symmetric. If $X$ is complex Banach space then clearly, $X^*$ is a Hilbert space. Whereas, if we consider $X$ as real Banach space then
For $\dim X\geq 3$ we have $X^*$ is a Hilbert space and consequently, $X$ is also a Hilbert space. Thus Birkhoff-James orthogonality in $X$ is symmetric. If $\dim X=2$ then $X^*$ is a Radon plane and so is $X$ (see Lemma \ref{Dual; Radon}). This implies Birkhoff-James orthogonality is symmetric in $X.$
\end{proof}

Since symmetry of Birkhoff-James orthogonality characterizes the inner product space having dimension more than two (Theorem \ref{symmetry}), we have the following consequence.
\begin{cor}
    Let $X$ be a  Banach space with dimension more than two. Suppose that either of the following holds true:
    \begin{itemize}
        \item [(i)] $X$ is finite-dimensional smooth.
        \item[(ii)] $X$ is finite-dimensional strictly convex.
        \item[(ii)] $X$ is infinite-dimensional ALUR.
    \end{itemize}
    Then  $X$ has $S0I$ if and only if the norm of $X$ is induced by an inner product. 
\end{cor}

\begin{remark}
    This shows that, for $1<p<\infty$, the space $\ell_p^n$ has the $S0I$ property if and only if $p=2$. We recall that, in Remark 3.8 of \cite{Bracic}, the authors conjectured that $\ell_p^n$ does not have the $S0I$ property for $p\neq 2$. Our result confirms this conjecture and, moreover, extends it to the general setting of $L_p(\mu)$ spaces for $1<p<\infty$.
\end{remark}
A natural next question is what happens in finite-dimensional polyhedral Banach spaces. In the next theorem, we show that no finite-dimensional polyhedral Banach space having dimension more than two has the $S0I$ property. We first prove the following lemma.

\begin{lemma}\label{lemma; positive}
    Let $X$ be a real finite-dimensional polyhedral Banach space and let $T\in L(X)$. Suppose that $y^*(Tx)>0$ for all $(x, y^*)\in \Pi \cap (\Ext B_X \times \Ext B_{X^*}).$ Then for any $(v, v^*)\in \Pi$, $v^*(Tv)>0.$  
\end{lemma}

\begin{proof}
Fix $(v,v^*)\in \Pi$. Since $B_X$ is a  polyhedron, there exist $x_1, x_2, \ldots, x_n \in \Ext B_X$ and $\alpha_i>0$, $\sum_{i=1}^n\alpha_i=1$, such that
$v=\sum_{i=1}^n\alpha_i x_i.$ Since $v^*\in J(v)$, we have $1=v^{*}(v)=\sum_{i=1}^n\alpha_i v^{*}(x_i).$
As $v^{*}(x_i)\leq1$ for every $i$, it follows that
$v^{*}(x_i)=1,\, i\in \{1,\ldots,n\}.$ Hence $v^{*}\in J(x_i)$ for every $i$. Since $J(x_i)$ is a  face of $B_{X^*}$, we can write $ v^{*}=\sum_{j=1}^{m_i}\beta_{ij}y_{ij}^*,$
where $\beta_{ij}\geq0$, $\sum_{j=1}^{m_i}\beta_{ij}=1$, and  $y_{ij}^*\in\Ext J(x_i)\subseteq\Ext B_{X^*}.$ Therefore,
\begin{eqnarray*}
v^{*}(Tv)&=&v^{*}\left(T\sum_{i=1}^n\alpha_i x_i\right)\\
&=&\sum_{i=1}^n\alpha_i v^{*}(Tx_i)\\
&=&\sum_{i=1}^n\alpha_i\sum_{j=1}^{m_i}\beta_{ij}y_{ij}^*(Tx_i).
\end{eqnarray*}
By the hypothesis, $y_{ij}^*(Tx_i)>0$ for every $i,j$. Since $\alpha_i,\beta_{ij}\geq0$ and $\sum_i\alpha_i=\sum_j\beta_{ij}=1$, we conclude that
$v^{*}(Tv)>0.$
\end{proof}

\begin{theorem}\label{polyhedral}
    Let $X$ be a real finite-dimensional polyhedral Banach space. Then $X$ has $S0I$ property if and only if Birkhoff-James orthogonality in $X$ is symmetric.
\end{theorem}

\begin{proof}
Since the sufficient part is trivial, we only prove the necessary part.  \\  Let $\Ext B_{X^*}=\{\pm y_1^*, \pm y_2^*, \ldots, \pm y_k^*\}.$ Suppose on the contrary there exists $x, y\in S_X$ such that $x\perp_B y$ but $y\not\perp_B x.$ Choose $y\in \sm(X)\cap S_X.$ Indeed, consider $x\in \Ext B_X.$ Suppose that $y_k^*\in \Ext J(x).$ If we take a smooth point $y$ in the corresponding facet $F_k:=\{z\in S_X: y_k^*(z)=1\}$ then clearly, $y\not\perp_B x.$  Let $J(y)=\{y_1^*\}$ and consider the corresponding facet $F_{1}=\{z\in S_X: y_1^*(z)=1\}.$  We assume that $y_1^*(x)=a>0.$ Define $T_1=y^*\otimes x$ and $T_2=I- y^*\otimes y.$  Then we want to find a $\lambda$ such that for $S_\lambda=T_1+\lambda T_2$, $v^*S_\lambda v\neq 0$ for all $(v, v^*)\in \Pi.$ From Lemma \ref{lemma; positive} it is enough to show that $y_{jk}^*S_\lambda v_j$ has the same sign for each  $v_j\in \Ext B_X$ and $y_{jk}^*\in \Ext J(v_j).$ Let $$M=\max\{|y_{ij}^*T_1v_i|: v_i\in \Ext B_X, \, y_{ij}^*\in \Ext(J(v_i))\}$$ and let $$c=\inf\{|y_{ij}^*T_2v_i|: v_i\in \Ext B_X, \, y_{ij}^*\in \Ext(J(v_i))\setminus \{\pm y_1^*\}\}.$$ Clearly, $c>0.$ We choose $\lambda > M/c>0$. Since for each $v_i\in \Ext B_X\cap (F_1\cup -F_1),$ we obtain $y_1^*S_\lambda v_i= a>0.$ Also, for any $v_i\in \Ext B_X, \, y_{ij}^*\in \Ext(J(v_i))\setminus \{\pm y_1^*\},$ we have $$|y_{ij}^*T_1v_i|\leq M = \frac{M}{c} c< \lambda |y_{ij}^*T_2v_i|.$$ This implies for any $v_i\in \Ext B_X, \, y_{ij}^*\in \Ext(J(v_i)),$ $y_{ij}^*S_\lambda v_i>0.$ This proves that for any $(v, v^*)\in \Pi$ we get $v^*S_\lambda  v\neq 0.$ Also, following the same argument as given in the proof of Theorem \ref{main} that $S_\lambda $ is invertible and considering $T=S_\lambda^{-1}$ we arrive at our desired result.
 \end{proof}

So the above result shows that for real finite-dimensional polyhedral Banach space with dimension more than two and can't have $S0I$ property. Moreover, we note the following corollary for two-dimensional Banach space using Theorem \ref{main} and Theorem \ref{polyhedral}. This also answers some questions posed by Poon in \cite{Poon}:
 \begin{cor}\label{cor; two-dim}
     Let $X$ be a real two-dimensional Banach space. If either of the following holds true
     \begin{itemize}
         \item [(i)] $X$ is smooth.
         \item[(ii)] $X$ is strictly convex.
         \item[(iii)] $X$ is polyhedral space.
     \end{itemize}
     Then $X$ has $S0I$ if and only if $X$ must be a Radon plane. Moreover, if we consider the scalar field is complex and $X$ is smooth or strictly convex, then $X$ has $S0I$ implies Birkhoff-James orthogonality is symmetric.
 \end{cor}

 In fact, we prove that for a two-dimensional real Banach space has $S0I$ property then it has to be a Radon plane. This answers the question raised by Poon \cite{Poon}.  Before we note the following lemma. We omit the proof because it follows directly rom the fact that the collection of nonsmooth points in a two-dimensional Banach space is atmost countable (see \cite[Ch. 2, p. 17]{Amir}).

 \begin{lemma}\label{Mazur}
     Suppose that $X$ is a  two-dimensional Banach space and let $x_1, x_2\in S_X$ such that $x_1\neq \pm x_2.$  Consider the following set: $$K=\bigg\{\frac{(1-t)x_1+tx_2}{\|(1-t)x_1+tx_2\|}: 0\leq t\leq 1\bigg\}.$$ Then $\sm(X)\cap K$ is dense in $K.$ 
 \end{lemma}

 \begin{theorem}
     Let $X$ be a two-dimensional real Banach space. Then $X$ has $S0I$ if and only if $X$ is a Radon plane.
 \end{theorem}

\begin{proof}
Since the sufficient part is clear, we only show that necessary part.\\
   If $X$ is smooth then it directly  follows from Corollary \ref{cor; two-dim}. So  assume that $X$ is nonsmooth. Suppose on the contrary that $X$ is not a Radon plane. Since every left symmetric element of $X$ is smooth,  We may consider $x\in \Ext B_X \setminus \sm (X)$ such that there exists $y\in S_X$ with $x\perp_By$ but $y\not\perp_B x$ with $y^*(x)=a>0.$  We  can choose $y\in \sm (X)\cap S_X.$ Indeed, One can easily check that the set  $A_x\cap x^\perp$ is open in $x^\perp$ with respect to the subspace topology. Now by Mazur's theorem we have smooth points are dense in $X$ and consequently from Lemma \ref{Mazur}, there exists $y\in A_x\cap x^\perp$ with $y\in \sm (X)\cap S_X.$    Consider  $T_1=y^*\otimes x$ and $T_2=I- y^*\otimes y.$  Let $\mathcal Z(T_2):=\{(v, v^*)\in \Pi: v^*T_2v=0\}.$   Now we proceed with the following two cases.\\
    \textbf{Case-I:} Suppose that $y$ is exposed. Observe that $\mathcal Z(T_2)=\{\pm(y, y^*)\in \Pi: J(y)=\{y^*\}\}.$ At $\pm (y, y^*),$ $y^*T_1y=y^*(x)=a>0.$ Also, for $T\in L(X)$, the map $\phi_T: \Pi\to \mathbb R$ defined by $\phi_T(x, x^*)=x^*Tx$ is continuous. Thus there exists a open set of $V=U\times U^* \subset \Pi$ with respect to product topology containing $\mathcal{Z}(T_2)$ such that for all $(v, v^*)\in V\cup (-V)$, $v^*T_1v>0.$ Moreover, for all $(v, v^*)\in \Pi \setminus V,$ we take $\inf |v^*T_2v|=c>0$ and $\sup |v^*T_1v|\leq M.$ Considering $\lambda > M/c>0$ and $S_\lambda=T_1+\lambda T_2$  we proceed with the same argument as given in the main theorem to obtain  $v^*S_\lambda v \neq 0$ for all $(v, v^*)\in \Pi.$\\
    \textbf{Case-II:} Now suppose that $y$ is not exposed. Since $y\in \sm (X)$ this implies there exists $v_1, v_2\in S_X$ such that $y=(1-t)v_1 + t v_2$ for some $0<t<1.$ Note that  $$\mathcal{Z}(T_2)=\{\pm(v, y^*): v\in F_{y^*}, J(y)=\{y^*\}\},$$ where $F_{y^*}=\{(1-t)v_1+tv_2: 0\leq t\leq 1\}.$ For $(v, y^*)\in \mathcal{Z}(T_2),$ $y^*T_1v=a>0.$  For each $v\in F_{y^*}$  there exists an open neighbourhood $U^*$ of  $y^*$ such that for all $z^*\in U^*$ and $(v, z^*)\in  \Pi,$ $z^*T_1v>0.$ Also, $z^*T_2v\geq0.$ On the other hand, note that the set $$K=\Pi \setminus \{(u, u^*): u=(1-t)v_1+tv_2, 0<t<1 \,\, \mbox{and}\,\, u^*\in U^*\}$$ is a compact set and therefore, for any $(v, v^*)\in \Pi \setminus K$ we may consider $\min |v^*T_2v|=c>0.$  Then taking $\lambda> \frac{\max|v^*T_1v|}{c}>0$ and $S_\lambda=T_1+\lambda T_2$ we get $v^*S_\lambda v \neq 0$ for all $(v, v^*)\in \Pi.$ Therefore, for $T=S_{\lambda}^{-1},$ we get $x\perp_B Tx$ whereas, $v\not\perp_B T^{-1}v$ for all $v\in S_X.$ This contradicts that $X$ has $S0I$.  Hence the theorem.
\end{proof}

\section{$S0I$ and local symmetry of operators}

In this section we mainly observe some connection between local $S0I$ property and local symmetry of operators, particularly in real finite-dimensional setting.  Here we provide a sufficient condition of $S0I$ at a point $T.$

\begin{theorem}
    Let $X$ be a finite-dimensional real Banach space and let $T$ be invertible. Suppose that $I\perp_B T\implies T\perp_B I.$ Then $X$ has $S0I$ at $T.$
\end{theorem}

\begin{proof}
    From norm derivatives, given $T, A\in L(X)$ we have the following functions:
    \begin{eqnarray*}
        \rho'_\pm(T, A)= \|T\|\lim_{t\to^{0\pm}} \frac{\|T+tA\|-\|T\|}{t}.
    \end{eqnarray*}
    From \cite{Wojcik}, we note that this can be written as: $\rho'_+(T, A)=\sup\{\rho'_+(Tx, Ax): x\in M_T\cap \Ext B_X\}$ and $\rho'_-(T, A)=\inf\{\rho'_-(Tx, Ax): x\in M_T\cap \Ext B_X\}.$ In our case we consider $A=I$ and denote $\rho'_+(T, I)=\rho_+$ and $\rho'_-(T, I)=\rho'_-.$  Suppose that $0\in W(T)$ and therefore, there exists $x\in S_X$ such that $x\perp_B Tx$ which implies that $I\perp_B T$ (see \cite{Bhatia}). Now following our hypothesis we have $T\perp_B I.$ This yields $\rho'_-\leq 0\leq \rho'_+.$ Since $S_X$ is  compact, there exists $x_0, y_0\in M_T\cap \Ext B_X$ and $u^*\in J(Tx_0), v^*\in J(Ty_0)$ such that $\rho'_+=u^*(x_0)$ and $\rho'_-=v^*(y_0),$ respectively. We note the following two cases:\\
    \textbf{Case-I:} Suppose that $\rho'_+=0$. Then $u^*(x_0)=0$. This implies $Tx_0\perp_B x_0.$ Similarly, if $\rho'_-=0$ we have $Ty_0\perp_B y_0.$ From Proposition \ref{prop1} we conclude that $X$ has $S0I$ at $T.$\\
    \textbf{Case-II:} Now suppose that $\rho'_-<0<\rho'_+.$ For each $x\in S_X$, we define the following two functions $h_+, h_-: S_X\to \mathbb{R}$ such that:
    \begin{eqnarray*}
        h_+(x)=\max_{y^*\in J(Tx)} y^*(x)\\
        h_-(x)= \min_{y^*\in J(Tx)} y^*(x).
    \end{eqnarray*}
    It is easy to see that for $h_-\leq h_+$. Also, $h_-(y_0)=\rho'_-$ and $h_+(x_0)=\rho'_+.$ We next show that the function $h_+$ is upper semi-continuous whereas $h_-$ is lower semi-continuous. We only show it for $h_+$ as the case for $h_-$ follows similarly. Suppose that there exists $\{x_n\}\subset S_X$ such that $x_n\to x.$ Since for each $n\in \mathbb{N},$ $J(Tx_n)$ is a compact set. Thus there exists $y_n^*\in J(Tx_n)$ such that $h_+(x_n)=y_n^*(x_n)$ for each $n.$ Let $y_n^*\to y^*$ for some $y^*\in S_{X^*}.$ A straightforward computation shows that $y^*\in J(Tx)$ and $y^*(x)\leq h_+(x).$ So it follows that $\lim \sup h_+(x_n)\leq h_+(x).$ This proves our claim. \\
     \textit{Subcase-I:} If $x_0=y_0$ then we have $h_+(x_0)=\rho'_+\geq 0$ and $h_-(x_0)=\rho'_-\leq 0.$ Since $J(Tx_0)$ is compact convex subset of $S_{X^*},$ there exists $h\in J(Tx_0)$ such that $h(x_0)=0.$ This shows that $Tx_0\perp_B x_0.$ \\
     \textit{Subcase-II:} Suppose that $x_0 \neq y_0.$ We have the following possibilities:
     \[
     (i) h_+(y_0)\geq 0. (ii) h_-(x_0)> 0.  (iii) h_-(x_0)\leq 0 \,\,\mbox{and},\,\, (iv) h_+(y_0)<  0.
    \]
     If we consider Possibility (i) with $h_-(y_0)\leq 0,$ we get $w^*(y_0)\geq 0$ and $z^*(y_0)\leq0$. By convexity of $J(Ty_0),$ we have $y^*\in J(Ty_0)$ such that $y^*(y_0)=0.$ This implies $Ty_0\perp_B y_0.$ Similarly,  for possibility-(iii) we observe that $Tx_0\perp_B x_0.$ So let us assume that we have Possibility-(ii), i.e., $h_-(x_0) > 0$. Let $\gamma:[0, 1]\to S_X$ be a path defined as $\gamma(0)=x_0$ and $\gamma(1)=y_0.$ Let $S:=\{t\in [0, 1]: h_-(\gamma(t))\leq 0 \}.$ Note that $S\neq \emptyset$ as $t=1\in S.$ Assume that $t_0=\inf S.$ For all $t<t_0$ we have $h_-(\gamma(t))>0.$ Observe that $t_0>0.$ Indeed, if possible let $\exists \{t_n\}\in [0, 1]$ such that $t_n\to 0$ and $h_-(\gamma(t_n))\leq 0$ for each $n.$ Since $\gamma $ is continuous $\gamma(t_n)\to \gamma (0)=x_0.$ By lower semi-continuity of $h_-$ we have $h_-(x_0)\leq \lim\inf h_-(\gamma(t_n))\leq 0.$ This contradicts $h_-(x_0)>0.$ Thus we have $t_0>0.$ Also, since $t_0=\inf S,$ one can observe that $h_-(\gamma(t_0))\leq 0.$ On the other hand, consider  $\{s_n\}\in [0, 1]$ such that $s_n< t_0$ and $s_n\to t_0.$ Then $\gamma(s_n)\to \gamma(t_0)$ and $h_+(\gamma(s_n))\geq h_-(\gamma(s_n))>0$ for each $n.$ Since $h_+$ is upper semi-continuous, it follows that $h_+(\gamma(t_0))\geq \lim\sup h_+(\gamma(s_n)) \geq 0.$ This shows that there exists $t_0\in [0, 1]$ such that $h_+(\gamma(t_0))\geq 0$ whereas $h_-(\gamma(t_0))\leq 0.$ Let $\gamma(t_0)=z_0.$ Therefore, again using the fact that $J(Tz_0)$ is a convex subset of $S_{X^*}$, there exists $ h\in J(Tz_0)$  such that $h(z_0)=0.$ This implies $z_0\perp_B Tz_0.$ This proves that $X$ has $S0I$ at $T$ and  completes the theorem.  
\end{proof}

From \cite[Chapter 7]{MPS} we note that $T\in L(\ell_1^n)$ is right symmetric if and only if $T$ is an extreme contraction of $B_{L(X, Y)}.$ Hence the following corollary is immediate. 
\begin{cor}
    Let $X$ be a finite-dimensional real Banach space and let $T\in GL(X)$ be right symmetric. Then $X$ has $S0I$ at $T.$  Moreover, if $X=\ell_1^n$ then any extreme contraction in $L(X)$ has $S0I$.
\end{cor}

\begin{remark}
    Of course, there exist operators that possess the $S0I$ property without being right symmetric. For instance, any invertible bounded linear operator on a Hilbert space that is not an isometry is not right symmetric (see \cite{Turnsek}), while it does possess the $S0I$ property. Although our theorem  tells that we only need right symmetry of $T$ with respect to $I$ to have $S0I$. 
\end{remark}

From \cite{KST}, we note that all the invertible elements are right symmetric within the positive cone of a unital $C^*$ algebra. Therefore, we observe the following.
\begin{cor}
    Suppose that $\mathcal{A}$ be a finite-dimensional unital real $C^*$ algebra and let $\mathcal{A}_+$ be the positive cone of $\mathcal{A}.$ Then for each invertible $A\in \mathcal{A}_+$, $\mathcal{A}$ has $S0I$ at $A.$
\end{cor}

\begin{cor}
    Let $X$ be a finite-dimensional real Banach space  and let $T\in L(X)$ be invertible. If the identity operator is local-$GL(X)$ left symmetric then $X$ has $S0I$.
\end{cor}

In the following example we observe that if $T$ has $S0I$ it does not necessarily imply that $I\perp_B T\implies T\perp_B I.$
\begin{example}
    Let $X$ be the real three dimensional Euclidean space and let $T\in L(X)$ be the following matrix with respect to standard ordered basis of $X,$
    \[T=\begin{pmatrix}
        0 & 0 & 1\\
        0 & 2 & 0\\
        1 & 0 & 0
    \end{pmatrix}.\] Clearly, $T$ is invertible. Hence $X$ has $S0I$ at $T$. Moreover, it is easy to see that $M_T=\{\pm e_2\}.$ First note that $e_1\perp_B Te_1$ and this implies $I\perp_B T,$ whereas $Te_2= 2e_2$ implies that $Te_2\not\perp_B e_2$ and thus $T\not\perp_B I.$ 
\end{example}

\noindent\textbf{Concluding remark:} To summarize, our results demonstrate that the $S0I$ property is a remarkably strong condition on the geometry of a Banach space. In several broad classes of Banach spaces, we have shown that the existence of the $S0I$ property necessarily forces the norm to be induced by an inner product. It is also worth noting that, throughout our investigation, we are unable to find any Banach space with the $S0I$ property for which Birkhoff-James orthogonality fails to be symmetric. This observation naturally leads us to the following stronger conjecture:

\textit{Suppose that $X$ is a normed space with the $S0I$ property. Then Birkhoff-James orthogonality in $X$ is symmetric.}

\noindent Whether this conjecture holds in full generality remains an intriguing question and, in our view, deserves further investigation. On the other hand, the collection of operators in $L(X)$ possessing the $S0I$ property forms a natural and interesting class of operators. A more local study of this class may provide further insight into the interplay between the geometry of the underlying Banach space, the operators acting on it, and their spatial numerical ranges.

\vspace{.5cm}

	\subsection*{Acknowledgements} I would like to acknowledge my supervisor Prof. Kallol Paul and my elder brother-cum-mentor Dr. Debmalya Sain for their immense contribution towards shaping the mathematical philosophy of numerous students including me.


\begin{thebibliography}{99}

     \bibitem{Amir} D. Amir, \textit{Characterizations of inner product spaces}, Operator Theory: Advances and Applications, Birkh\"auser, Basel, \textbf{20} (1986).
		
		 \bibitem{Birkhoff35} G. Birkhoff, \textit{Orthogonality in linear metric spaces}, Duke Math. J., \textbf{1} (1935) 169–172.
		
		\bibitem{Bracic} J. Bra\v{c}i\v{c},  C. Diogo,  \textit{Simultaneous zero inclusion property for spatial numerical ranges}, J. Math. Anal. Appl. 449 (2017) 1413–1423.
		
		\bibitem{Bhatia} R. Bhatia and P. Šemrl, \textit{Orthogonality of matrices and some distance problems}, Linear Algebra Appl., 287 (1999), 77–85.

        \bibitem{CSS} A. Chattopadhyay, D. Sain and T. Senapati,\textit{Characterization of symmetric points in $\ell_p^n$
spaces}, Linear Multilinear Algebra, \textbf{69} (2021), 2998-3009.

        \bibitem{Day47} M.M. Day, \textit{Some characterizations of inner-product spaces}, Trans. Am. Math. Soc., \textbf{62} (1947) 320–337.


        \bibitem{James45} R.C. James,\textit{ Orthogonality in Banach spaces}, Duke Math. J., \textbf{12} (1945) 291–302.
        
       \bibitem{James47a} R.C. James, \textit{Inner products in Banach spaces}, Bull. Am. Math. Soc., \textbf{53} (1947), 559–566.
       
      \bibitem{James47b} R.C. James,\textit{ Orthogonality and linear functionals in Banach spaces}, Trans. Am. Math. Soc.,\textbf{ 61 }(1947) 265–292.

      
	      \bibitem{KST}   N. Komuro, K-S. Saito, R. Tanaka,     \textit{On symmetry of Birkhoff orthogonality in the positive cones of C⁎-algebras with applications}, J. Math. Anal. Appl.,  \textbf{474} (2019), 1488–1497.

      

        
		
	
		
		
		\bibitem{MPS} A. Mal, K. Paul and D. Sain, \textit{Birkhoff–James Orthogonality and Geometry  of Operator Spaces}, Infosys Science Foundation Series, Springer, Singapore, 2024.
		
		\bibitem{Poon} E. Poon, \textit{Birkhoff-James orthogonality and the simultaneous zero inclusion property}, J. Math. Anal. Appl., \textbf{ 534} (2024) 128052.


		\bibitem{Sain17} D. Sain, \textit{Birkhoff–James orthogonality of linear operators on finite dimensional Banach spaces}, J. Math. Anal. Appl., \textbf{447} (2017), 860-866.


\bibitem{Turnsek} A. Turn\v sek,\textit{ On operators preserving James’ orthogonality}, Linear Algebra Appl., 407 (2005), 189–195.
        \bibitem{Wojcik} P. W\'ojcik, \textit{Birkhoff orthogonality in classical $M$-ideals}, J. Aust. Math. Soc., \textbf{103} (2017), 279–288.
		
		
	\end{thebibliography}
\end{document}